\documentclass[a4paper,11pt]{amsart} 
\usepackage{amsmath,amsxtra,amssymb,latexsym, amscd,amsthm}
\usepackage[mathscr]{eucal}
\usepackage{mathrsfs}
\usepackage{bbm}
\usepackage{enumerate}
\usepackage{pict2e}
\usepackage{tikz}
\usepackage{graphicx}
\usepackage[a4paper]{geometry}
\usepackage{color}
\usepackage{hyperref}

\theoremstyle{plain}
\newtheorem{thm}{Theorem}[section]

\newtheorem{lem}[thm]{Lemma}
\theoremstyle{definition}

\newtheorem*{rem*}{Remark}

\newcommand{\dd}{\mathrm{d}}
\newcommand{\ee}{\mathrm{e}}
\newcommand{\ii}{\mathrm{i}}

\newcommand{\Z}{\mathbb{Z}}

\newcommand{\T}{\mathbb{T}}

\newcommand{\cA}{\mathcal{A}}

\DeclareSymbolFont{extraup}{U}{zavm}{m}{n}
\DeclareMathSymbol{\varheart}{\mathalpha}{extraup}{86}
\DeclareMathSymbol{\vardiamond}{\mathalpha}{extraup}{87}

\title{Dispersion on certain Cartesian products of graphs}
\author{Ka\"is Ammari and Mostafa Sabri}
\address{LR Analysis and Control of PDEs, LR22ES03, Department of Mathematics, Faculty of
Sciences of Monastir, University of Monastir, 5019 Monastir, Tunisia.}
\email{kais.ammari@fsm.rnu.tn}
\address{Department of Mathematics, Faculty of Science, Cairo University, Giza 12613, Egypt.}
\email{mmsabri@sci.cu.edu.eg}
\usepackage{calc}
\usepackage{graphicx}

\makeatletter
\newlength{\temp@wc@width}
\newlength{\temp@wc@height}
\newcommand{\widecheck}[1]{%
  \setlength{\temp@wc@width}{\widthof{$#1$}}%
  \setlength{\temp@wc@height}{\heightof{$#1$}}%
  #1\hspace{-\temp@wc@width}%
  \raisebox{\temp@wc@height+2pt}[\heightof{$\widehat{#1}$}]%
     {\rotatebox[origin=c]{180}{\vbox to 0pt{\hbox{$\widehat{\hphantom{#1}}$}}}}%
}
\makeatother

\begin{document}

\begin{abstract}
In this short note we prove a sharp dispersive estimate $\|\ee^{\ii tH} f\|_\infty < t^{-d/3}\|f\|_1$ for any Cartesian product $\Z^d\mathop\square G_F$ of the integer lattice and a finite graph. This includes the infinite ladder, $k$-strips and infinite cylinders, which can be endowed with certain potentials.
\end{abstract}

\maketitle

\section{Introduction and main result} 

In this note we discuss dispersion on Cartesian products of the form $\Z^d\mathop\square G_F$, where $G_F$ is a finite graph. The proof of the estimate is quite simple, it combines known estimates on the Bessel function with the Floquet theory of periodic graphs. However, to our knowledge, the phenomenon has not been observed before in the literature.

\medskip

Recall that the Cartesian product $G\mathop\square H$ of two graphs $G$ and $H$ is the graph with vertex set $V(G)\times V(H)$, such that $(u,v)\sim (u',v')$ iff ($u=u'$ and $v\sim v'$) or ($u\sim u'$ and $v=v'$).

\medskip

The product $\Z\mathop\square G_F$ is very easy to visualize. Simply take the integer line $\Z$, replace each vertex by a copy of $G_F$, then connect the matching vertices. If $G_F=P_2$, the $2$-path, this creates an infinite ladder. More generally, if $G_F=P_k$ is a $k$-path, we get an infinite strip of width $k$. 

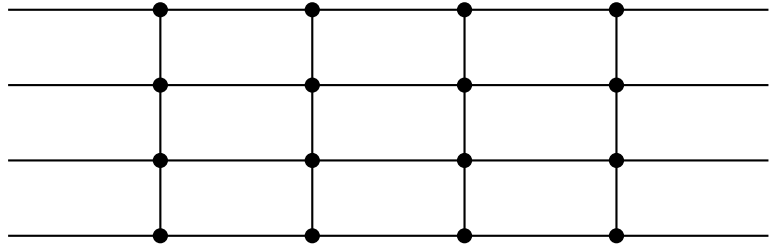
\begin{figure}[h!]
\begin{center}
\setlength{\unitlength}{1cm}
\thicklines
\begin{picture}(3,3)(-2,-2)
   \put(-5,-2){\line(1,0){10}}
	 \put(-5,-1){\line(1,0){10}}
	 \put(-5,0){\line(1,0){10}}
	 \put(-5,1){\line(1,0){10}}
	 \put(-3,-2){\line(0,1){3}}
	 \put(-1,-2){\line(0,1){3}}
	 \put(1,-2){\line(0,1){3}}
	 \put(3,-2){\line(0,1){3}}
	 \put(-1,-1){\circle*{.2}}
	 \put(-1,0){\circle*{.2}}
	 \put(-3,-1){\circle*{.2}}
	 \put(-3,0){\circle*{.2}}
	 \put(1,-1){\circle*{.2}}
	 \put(1,0){\circle*{.2}}
	 \put(3,-1){\circle*{.2}}
	 \put(3,0){\circle*{.2}}
	 \put(-3,1){\circle*{.2}}
	 \put(-1,1){\circle*{.2}}
	 \put(1,1){\circle*{.2}}
	 \put(3,1){\circle*{.2}}
	 \put(-3,-2){\circle*{.2}}
	 \put(-1,-2){\circle*{.2}}
	 \put(1,-2){\circle*{.2}}
	 \put(3,-2){\circle*{.2}}
\end{picture}
\caption{The $4$-strip, $\Z\mathop\square P_4$.}\label{fig:stri}
\end{center}
\end{figure}

If we take $G_F=C_p$, a $p$-cycle, we get an infinite cylinder.

\begin{figure}[h!]
\begin{center}
\setlength{\unitlength}{1cm}
\thicklines
\begin{picture}(1.3,1.3)(-1.3,-1.3)
   \put(-5,0){\line(1,0){10}}
	 \multiput(-2,-0.5)(-0.2,-0.1){10}{\line(-1,-0.5){0.1}}
 	 \multiput(0,-0.5)(-0.2,-0.1){10}{\line(-1,-0.5){0.1}}
	 \multiput(2,-0.5)(-0.2,-0.1){10}{\line(-1,-0.5){0.1}}
	 \multiput(4,-0.5)(-0.2,-0.1){10}{\line(-1,-0.5){0.1}}
	 \multiput(0,-0.5)(-0.2,0.1){5}{\line(-1,0.5){0.1}}
	 \multiput(2,-0.5)(-0.2,0.1){5}{\line(-1,0.5){0.1}}
	 \multiput(4,-0.5)(-0.2,0.1){5}{\line(-1,0.5){0.1}}
	 \multiput(-2,-0.5)(-0.2,0.1){5}{\line(-1,0.5){0.1}}
	 \put(-1,0){\circle*{.2}}
	 \put(-3,0){\circle*{.2}}
	 \put(1,0){\circle*{.2}}
	 \put(3,0){\circle*{.2}}
	\put(-6,-1.5){\line(1,0){10}}
	\multiput(-4,-0.5)(0.4,0){22}{\line(1,0){0.2}}
	\put(-2,-1.5){\circle*{.2}}
	\put(-4,-1.5){\circle*{.2}}
	\put(0,-1.5){\circle*{.2}}
	\put(2,-1.5){\circle*{.2}}
	\put(-2,-0.5){\circle{.2}}
	\put(0,-0.5){\circle{.2}}
	\put(2,-0.5){\circle{.2}}
	\put(4,-0.5){\circle{.2}}
	\put(2,-1.5){\line(1,1.5){1}}
	\put(0,-1.5){\line(1,1.5){1}}
	\put(-2,-1.5){\line(1,1.5){1}}
	\put(-4,-1.5){\line(1,1.5){1}}
\end{picture}
\caption{The graph $\Z\mathop\square C_3$.}\label{fig:cyl}
\end{center}
\end{figure}
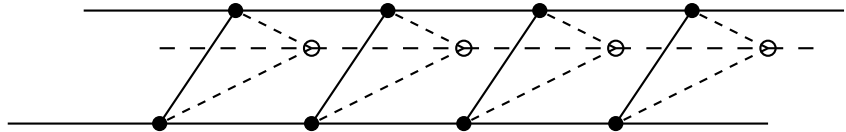

We can similarly construct $\Z^d\mathop\square G_F$ by replacing each vertex in $\Z^d$ by a copy of $G_F$ and connecting the matching vertices. For example, $\Z^2\mathop\square P_2$ consists of an infinite layer of cubes. Our results apply to all these graphs.

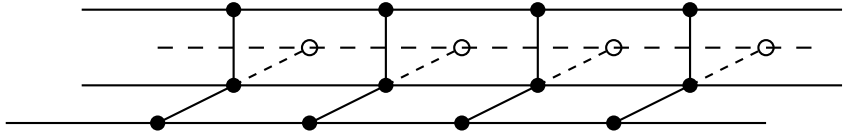
\begin{figure}[h!]
\begin{center}
\setlength{\unitlength}{1cm}
\thicklines
\begin{picture}(1.3,1.3)(-1.3,-1.3)
   \put(-5,0){\line(1,0){10}}
	 \put(-5,-1){\line(1,0){10}}
	 \put(-3,-1){\line(0,1){1}}
	 \put(-1,-1){\line(0,1){1}}
	 \put(1,-1){\line(0,1){1}}
	 \put(3,-1){\line(0,1){1}}
	 \put(1,-1){\line(-1,-0.5){1}}
	 \put(3,-1){\line(-1,-0.5){1}}
 	 \put(-1,-1){\line(-1,-0.5){1}}
	 \put(-3,-1){\line(-1,-0.5){1}}
	\multiput(-1,-1)(0.2,0.1){5}{\line(1,0.5){0.1}}
	\multiput(1,-1)(0.2,0.1){5}{\line(1,0.5){0.1}}
	\multiput(3,-1)(0.2,0.1){5}{\line(1,0.5){0.1}}
	\multiput(-3,-1)(0.2,0.1){5}{\line(1,0.5){0.1}}
	 \put(-1,-1){\circle*{.2}}
	 \put(-1,0){\circle*{.2}}
	 \put(-3,-1){\circle*{.2}}
	 \put(-3,0){\circle*{.2}}
	 \put(1,-1){\circle*{.2}}
	 \put(1,0){\circle*{.2}}
	 \put(3,-1){\circle*{.2}}
	 \put(3,0){\circle*{.2}}
	\put(-6,-1.5){\line(1,0){10}}
	\multiput(-4,-0.5)(0.4,0){22}{\line(1,0){0.2}}
	\put(-2,-1.5){\circle*{.2}}
	\put(-4,-1.5){\circle*{.2}}
	\put(0,-1.5){\circle*{.2}}
	\put(2,-1.5){\circle*{.2}}
	\put(-2,-0.5){\circle{.2}}
	\put(0,-0.5){\circle{.2}}
	\put(2,-0.5){\circle{.2}}
	\put(4,-0.5){\circle{.2}}
\end{picture}
\caption{The graph $\Z\mathop\square \star_3$, where $\star_k$ is the star graph with $k$ edges.}\label{fig:star}
\end{center}
\end{figure}

We can furthermore add a potential $Q$ on $G_F$ which is copied across the layers (i.e. $Q(n+v_p)=Q(v_p)$). In this way we obtain a Schr\"odinger operator $H_{\Z^d\mathop\square G_F}$ with a periodic potential.

\medskip

Denote the points on $\Z^d\mathop\square G_F$ by $v = n+v_p$, where $n\in \Z^d$ and $v_p\in G_F$.

\medskip

We may now state our main result. Let $H_{G_F} = \cA_{G_F}+Q$.

\begin{thm} \label{main}
We have for $H=H_{\Z^d\mathop\square G_F}$,
\begin{equation}\label{e:kermain}
\ee^{\ii tH}(n+v_p,m+v_q) = \Big(\prod_{j=1}^d \ii^{n_j-m_j} J_{n_j-m_j}(2t) \Big)  \ee^{\ii t H_{G_F}}(v_p,v_q)\,,
\end{equation}
where $J_k(t)$ is the Bessel function. This implies a dispersion bound of the form
\begin{equation}\label{e:dismain}
\|\ee^{\ii t H}f\|_\infty < \frac{1}{t^{d/3}} \|f\|_1 \qquad \forall \, t > 0
\end{equation}
for $f\in \ell^1(\Z^d\mathop\square G_F)$, which is sharp.
\end{thm}

The Schr\"odinger operator $H$ has purely absolutely continuous spectrum and is thus spectrally delocalized (see Section~\ref{sec:pro}). The transport is also ballistic, so waves travel at maximum speed \cite{BMS}. Here \eqref{e:dismain} shows that as the waves travel, they ``flatten out''. In fact, since $\ee^{\ii tH}$ is unitary, $\|\ee^{\ii tH} f\|_2=\|f\|_2$ is conserved, so the only way that the supremum norm can decay with time is that the wave spreads out. Dispersion excludes the possibility that $|\ee^{\ii tH} f(x)|^2$ is a ``sliding bump'' of the form $\psi(x-t)$ for some $\psi\ge 0$ of compact support.

\medskip

We mention that $H$ is also spatially delocalized in the sense that its (generalized) eigenvectors have a kind of equidistribution in space. More precisely, quantum ergodicity holds partially on large finite subgraphs. See \cite{MS} for details in a more general framework.

\medskip

Dispersion on (discrete) graphs has not been extensively studied to our knowledge. We mention the adjacency matrix on $\Z^d$ in \cite{SK} with sharp speed $t^{-d/3}$, on the $(q+1)$-regular tree $\T_q$ with sharp speed $t^{-3/2}$ \cite{AS,E,S}, the integer line $\Z$ with a periodic potential taking $2$ values \cite{MZ}, more recently taking $k$ values \cite{MZ2}, with upper bounds $t^{-1/3}$ and $t^{-1/(k+1)}$ on the speed, respectively. The paper \cite{EKT} considers the case of $\Z$ with an $\ell^1$ potential and derive the speed $t^{-1/3}$. Finally, \cite{IS} considers two coupled discrete Schr\"odinger equations on $\Z$ and obtain again the speed $t^{-1/3}$.

\begin{figure}[h!]
\begin{center}
\setlength{\unitlength}{1cm}
\thicklines
\begin{picture}(1,1)(-1,-1)
   \put(-5,0){\line(1,0){10}}
	 \put(-5,-1){\line(1,0){10}}
	 \put(-3,-1){\line(0,1){1}}
	 \put(-1,-1){\line(0,1){1}}
	 \put(1,-1){\line(0,1){1}}
	 \put(3,-1){\line(0,1){1}}
	 \put(-1,-1){\circle*{.2}}
	 \put(-1,0){\circle{.2}}
	 \put(-3,-1){\circle*{.2}}
	 \put(-3,0){\circle{.2}}
	 \put(1,-1){\circle*{.2}}
	 \put(1,0){\circle{.2}}
	 \put(3,-1){\circle*{.2}}
	 \put(3,0){\circle{.2}}
\end{picture}
\caption{The ladder graph, $\Z\mathop\square P_2$, endowed two potentials $Q_\bullet,Q_\circ$.}\label{fig:lad}
\end{center}
\end{figure}
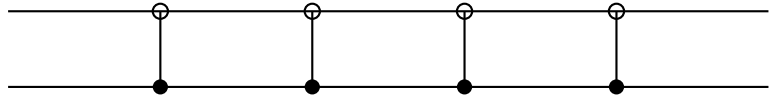

The results of \cite{MZ,MZ2} can be compared to those of a ladder graph (Fig.~\ref{fig:lad}) and a $k$-strip, endowed periodic potentials $Q_1,\dots,Q_k$ coming in \emph{parallel sheets}. Our proof shows that it is easier to prove dispersion on these graphs than on the simple one-dimensional line $\Z$ taking \emph{successive} potentials $Q_1,\dots,Q_k$.

\section{Proof of Theorem \ref{main}}\label{sec:pro}

Endow the finite graph $G_F$ with a potential $Q$ which is copied across the $G_F$ layers. By construction, the graph $\Gamma=\Z^d\mathop\square G_F$ is periodic, with fundamental crystal $G_F$. In other words, the Cartesian product is obtained by translating the fundamental domain $V_f=G_F$ under the action of $\Z^d$. This is a special case of the periodic graphs considered in \cite{KorSa}. As such, the Floquet theory applies to them. Let us discuss this in more detail.

\medskip

Due to translation invariance, we have
\[
\Gamma = \Z^d+G_F
\]
so that points in $\Gamma$ take the form $k+v_n$ for some $k\in \Z^d$ and $v_n\in G_F$. Here $k$ can be regarded as an ``integer part'' while $v_n$ a ``fractional part'' in the fundamental domain.

\medskip

Let us define the operator $U:\ell^2(\Gamma)\to \int_{\T_\ast^d}^\oplus \ell^2(G_F) \dd\theta$, where $\T_\ast^d=[0,1)^d$, by 
\[
(U\psi)_\theta(v_n) = \sum_{k\in \Z^d} \ee^{-2\pi \ii\theta\cdot k} \psi(k+v_n) \,.
\]
Then $U$ is unitary, with inverse $U^{-1}:(g_\theta)\mapsto \int_{\T_\ast^d} g_\theta(v_n)\ee^{2\pi\ii k\cdot \theta}\,\dd \theta$, and
\begin{equation}\label{e:uniquev}
UHU^{-1} = \int_{\T_\ast^d}^\oplus H(\theta)\,\dd\theta\,,
\end{equation}
where
\[
H(\theta)f(v_n) = \sum_{u\sim v_n}\ee^{2\pi \ii\theta\cdot \lfloor u\rfloor} f(\{u\}) + Q(v_n)f(v_n)\,,
\]
see \cite[\S 3.2]{BMS} for details.\footnote{In \cite{BMS}, the unitary operator is instead $(\widetilde{U}\psi)_\theta(v_n)=\ee^{-2\pi\ii\theta\cdot v_n}(U\psi)_\theta(v_n)$, so the fiber operator obtained there is $\widetilde{H}(\theta) = \ee^{-2\pi\ii\theta\cdot}H(\theta)\ee^{2\pi\ii\theta\cdot}$, where $\ee^{\pm 2\pi\ii\theta\cdot}f(v_p)=\ee^{2\pi\ii\theta\cdot v_p}f(v_p)$. This operator and ours are thus unitarily equivalent, the present version is just more suitable for our computations.} Here $\lfloor k+v_p\rfloor = k$ and $\{k+v_p\} = v_p$ are the integer and fractional parts, respectively, and the sum runs over the neighbors of $v_n$ in all $\Gamma$, not just $G_F$.

\medskip

By definition of the Cartesian product, the neighbors of $v_n = v_n+0$ are of two kinds~: those of the form $v_n\pm \mathfrak{e}_j$, where $\mathfrak{e}_j$ is the standard basis of $\Z^d$ (the neighbors of $0$ in $\Z^d$) and those of the form $v_p+0=v_p$, where $v_p\sim v_n$ in $G_F$. The operator thus simplifies to
\[
H(\theta)f(v_n) = \Big(\sum_{j=1}^d 2\cos 2\pi\theta_j\Big)f(v_n) + \sum_{v_p\sim v_n} f(v_p) + Q(v_n) f(v_n)\,,
\]
where we used that $\lfloor v_n \pm \mathfrak{e}_j\rfloor = \pm \mathfrak{e}_j$, $\{v_n\pm \mathfrak{e}_j\}=v_n$, $\lfloor v_p\rfloor = 0$, $\{v_p\}=v_p$.

\medskip

Thus, if $H_{G_F} = \cA_{G_F}+Q$ is the Schr\"odinger operator of the finite graph $G_F$, then in our setting, the fiber operator $H(\theta)$ is simply
\[
H(\theta) = \Big(\sum_{j=1}^d 2\cos(2\pi\theta_j)\Big)\cdot \mathrm{Id} + H_{G_F}\,.
\]
In other words, $H(\theta)$ is just $H_{G_F}$ shifted by a scalar. The Floquet eigenvalues take the form $E_s(\theta)=(\sum_{j=1}^d 2\cos(2\pi\theta_j))+\mu_s$, where $\mu_s$ are the eigenvalues of $H_{G_F}$, in particular $E_s(\theta)$ are analytic and non-constant, so the spectrum is purely absolutely continuous (see \cite[Section XIII.16]{RS4}). Moreover,
\begin{equation}\label{e:expfor}
\ee^{\ii t H(\theta)} = \ee^{2\ii t\sum_{j=1}^d \cos(2\pi\theta_j)}\ee^{\ii tH_{G_F}}\,.
\end{equation}

Finally, by \eqref{e:uniquev},
\begin{align*}
\ee^{\ii t H}(n+v_p,m+v_q) &= \langle U\delta_{n+v_p}, U\ee^{\ii tH}U^{-1}U\delta_{m+v_q}\rangle \\
&= \int_{\T_\ast^d}\sum_{v_r\in G_F} \overline{(U\delta_{n+v_p})_\theta(v_r)}(\ee^{\ii tH(\theta)}U\delta_{m+v_q})_\theta(v_r)\,\dd\theta \,.
\end{align*}
But $$(U\delta_{n+v_p})_\theta(v_r) = \sum_{k\in \Z^d} \ee^{-2\pi\ii\theta\cdot k}\delta_{n+v_p}(k+v_r) = \ee^{-2\pi\ii\theta\cdot n}\delta_{v_p,v_r}.$$ Hence,
\begin{align*}
\ee^{\ii t H}(n+v_p,m+v_q) &= \int_{\T_\ast^d} \sum_{v_s\in G_F}\ee^{2\pi\ii\theta\cdot n} \ee^{\ii tH(\theta)}(v_p,v_s)(U\delta_{m+v_q})(v_s)\,\dd\theta\\
&=\int_{\T_\ast^d} \ee^{2\pi\ii \theta\cdot(n-m)}\ee^{\ii tH(\theta)}(v_p,v_q)\,\dd\theta \,.
\end{align*}

Recalling \eqref{e:expfor}, this becomes
\[
\ee^{\ii t H}(n+v_p,m+v_q) = \left(\int_{\T_\ast^d}\ee^{2\pi\ii\theta\cdot (n-m)}\ee^{2\ii t\sum_{j=1}^d \cos(2\pi\theta_j)}\,\dd\theta\right)\ee^{\ii t H_{G_F}}(v_p,v_q)\,,
\]
which is \eqref{e:kermain}, since $\int_0^1 \ee^{2\pi \ii \nu x}\ee^{\ii t \cos(2\pi x)}\,\dd x = \ii^\nu J_\nu(t)$.

\medskip

This proves our main formula. From here dispersion follows from known estimates. In fact, by \cite{Landau}, we know that $|J_\nu(t)|\le c t^{-1/3}$, with $c\approx 0.786$ and that this estimate is sharp in the speed $t^{-1/3}$ (and in the constant). On the other hand, $|\ee^{\ii tH_{G_F}}(v_p,v_q)|\le \|\ee^{\ii tH_{G_F}}\| =1$. Since
\[
\|\ee^{\ii tH}f\|_\infty = \sup_{v\in \Gamma} \Big|\sum_{w\in \Gamma} \ee^{\ii tH}(v,w)f(w)\Big|\le \sup_{v,w\in \Gamma}|\ee^{\ii tH}(v,w)|\cdot \|f\|_1\,,
\]
the theorem follows.
\qed

\medskip

Note that the term $\ee^{\ii t H_{G_F}}(v_p,v_q)$ cannot improve dispersion as it is the evolution kernel of a Schr\"odinger operator on a finite graph, in particular has point spectrum and is dynamically localized. More precisely, if we had $\max_{v_p,v_q}|\ee^{\ii tH_{G_F}}(v_p,v_q)|\le \frac{c}{f(t)}$ for some $f(t)\to \infty$, this would imply $\|\ee^{\ii t H_{G_F}} \psi\|_\infty \le \frac{c}{f(t)}\|\psi\|_1$, and this would contradict the following lemma. This shows that the speed $t^{-d/3}$ cannot be improved by varying $G_F$.

\begin{lem}
There is no dispersion on finite graphs. More precisely, if $\|\psi\|_2=1$, it is impossible to find $f(t)\to \infty$ such that $\|\ee^{\ii tH_{G_F}} \psi\|_\infty \le \frac{c}{f(t)}\|\psi\|_1$.
\end{lem}
\begin{proof}
Let $\psi_t = \ee^{\ii tH_{G_F}}\psi$. Suppose on the contrary that $\|\psi_t\|_\infty \le \frac{c}{f(t)}\|\psi\|_1$. Let $0<\varepsilon<1$. By the RAGE theorem \cite{Teschl}, we can find a compact $K$ such that $\sup_t\|\chi_{K^c}\psi_t\|<\varepsilon$. So $\|\chi_K\psi_t\|^2 = \|\psi_t\|^2-\|\chi_{K^c}\psi_t\|^2\ge 1-\varepsilon^2$. But
\[
\|\chi_K\psi_t\|^2\le \|\psi_t\|_\infty\cdot |K| \le c\frac{\|\psi\|_1}{f(t)}\cdot |K|\,.
\]
It follows that $|K|\ge c'f(t)(1-\varepsilon^2)$. Taking $t\to\infty$ yields a contradiction since $K$ is compact. 
\end{proof}

\providecommand{\bysame}{\leavevmode\hbox to3em{\hrulefill}\thinspace}
\providecommand{\MR}{\relax\ifhmode\unskip\space\fi MR }
\providecommand{\MRhref}[2]{%
  \href{http://www.ams.org/mathscinet-getitem?mr=#1}{#2}
}
\providecommand{\href}[2]{#2}

\end{document}